\documentclass{amsart}

\usepackage[utf8]{inputenc}

\usepackage{amsmath,amsthm,amssymb,mathtools,xcolor,graphicx}
\usepackage{tikz}
\usetikzlibrary{shapes}

\usepackage{hyperref}
\hypersetup{
    colorlinks=true,
    linkcolor=blue,
    citecolor=blue
}
\usepackage{cleveref}

\usepackage{fp}
\usepackage{rotating}
\usetikzlibrary{arrows}
\usepackage{longtable}
\usepackage{mathtools}
\usepackage{centernot}

\title{On basic double G-links of squarefree monomial ideals}

\author{Patricia Klein}
\address{Texas A\&M University, Department of Mathematics, College Station, TX, USA}
\email{pjklein@tamu.edu}

\author{Matthew Koban}
\address{University of Toronto, Department of Mathematics, Toronto, ON, Canada}
\email{matthew.koban@mail.utoronto.ca}

\author{Jenna Rajchgot}
\address{McMaster University, Department of Mathematics and Statistics, Hamilton, ON, Canada}
\email{rajchgot@math.mcmaster.ca}

\newtheorem{theorem}{Theorem}[section]
\newtheorem*{theorem*}{Theorem}
\newtheorem{proposition}[theorem]{Proposition}
\newtheorem{example}[theorem]{Example}
\newtheorem{definition}[theorem]{Definition}
\newtheorem{lemma}[theorem]{Lemma}

\newtheorem*{claim*}{Claim}
\newtheorem{corollary}[theorem]{Corollary}
\newtheorem{Main Conjecture}[theorem]{Main Conjecture}

\theoremstyle{remark}
\newtheorem{Remark}[theorem]{Remark}

\newtheorem{question}[theorem]{Question}

\newcommand{\hgt}{{\rm{ht}}}

\newcommand{\PIVAL}{3.14159265358979323846264338}
\newcounter{i} 
\newcommand{\Circulant}[2] { 
	\begin{tikzpicture}[scale = 1.5]
	\setcounter{i}{0}
	\whiledo{\value{i}<#1}{ 
		\FPmul\tempA{2}{\thei} 
		\FPdiv\tempB{\PIVAL}{#1} 
		\FPmul\tempC{\tempA}{\tempB} 
		\FPcos\varX{\tempC} 
		\FPsin\varY{\tempC} 
		\stepcounter{i} 
		\FPround\varX{\varX}{3}
		\FPround\varY{\varY}{3}
		\node (\thei) at (\varX,\varY)[place]{ }; 
		\foreach \x in {#2} { 
			\pgfmathparse{mod(\x+\thei,#1)} 
			\let\tempB\pgfmathresult
			\pgfmathparse{mod(\thei-\x,#1)} 
			\let\tempA\pgfmathresult
			\ifthenelse{\lengthtest{\tempA pt < 1 pt}}{\FPadd\tempA{\tempA}{#1}}{}
			\ifthenelse{\lengthtest{\tempB pt < 1 pt}}{\FPadd\tempB{\tempB}{#1}}{}
			\ifthenelse{\lengthtest{\tempA pt > \thei pt}}{}{\ifthenelse{\thei = \tempA}{}{\draw [] (\thei) to (\tempA)}};
			\ifthenelse{\lengthtest{\tempB pt > \thei pt}}{}{\ifthenelse{\thei = \tempB}{}{\draw [] (\thei) to (\tempB)}};
		}
	}
	\end{tikzpicture}
}

\begin{document}

\nocite{*}

\tikzstyle{place}=[draw,circle,minimum size=0.5mm,inner sep=1pt,outer sep=-1.1pt,fill=black]

\maketitle

\begin{abstract}

Nagel and R\"omer introduced the class of weakly vertex decomposable simplicial complexes, which include matroid, shifted, and Gorenstein complexes as well as vertex decomposable complexes.  They proved that the Stanley--Reisner ideal of every weakly vertex decomposable simplicial complex is Gorenstein linked to an ideal of indeterminates via a sequence of basic double G-links.  In this paper, we explore basic double G-links between squarefree monomial ideals beyond the weakly vertex decomposable setting.  

Our first contribution is a structural result about certain basic double G-links which involve an edge ideal. Specifically, suppose $I(G)$ is the edge ideal of a graph $G$.  When $I(G)$ is a basic double G-link of a monomial ideal $B$ on an arbitrary homogeneous ideal $A$, we give a generating set for $B$ in terms of $G$ and show that this basic double G-link must be of degree $1$. Our second focus is on examples from the literature of simplicial complexes known to be Cohen--Macaulay but not weakly vertex decomposable. We show that these examples are not basic double G-links of any other squarefree monomial ideals.

\end{abstract}

\section{Introduction}

Broadly speaking, Gorenstein liaison is a framework for studying which properties of one subscheme of the projective space $\mathbb{P}^n$ can be transferred to another when their union is sufficiently nice, that is, when the subschemes are \emph{Gorenstein linked} (or \emph{G-linked}). 
Gorenstein links generate an equivalence relation whose equivalence classes are called \emph{Gorenstein liaison classes}.  Two subschemes in the same Gorenstein liaison class have the same codimension, and their homogeneous coordinate rings have the same depth.  Hence, every subscheme that is in the \emph{Gorenstein liaison class of a complete intersection} (abbreviated \emph{glicci}) is arithmetically Cohen--Macaulay.  An important open question in Gorenstein liaison is whether every arithmetically Cohen--Macaulay subscheme is glicci.  

This question has garnered a great deal of interest, and there have been a number of important partial results.  For example, Casanellas, Drozd, and Hartshorne \cite{CDH05} showed that every arithmetically Gorenstein subscheme is glicci.  Gorla \cite{Gor08} showed that every standard determinantal scheme is glicci, thereby generalizing results of \cite{Har07} and \cite{Memoir}. Migliore and Nagel \cite{MN02} showed that every arithmetically Cohen-Macaulay subscheme of $\mathbb{P}^n$, when viewed instead as a subscheme of $\mathbb{P}^{n+1}$, is glicci. 

Many specific classes of arithmetically Cohen--Macaulay schemes have also been proved to be glicci. Examples include certain curves in $\mathbb{P}^4$ \cite{CM00, CM01}, various generalized determinantal and Pfaffian schemes \cite{Gor07, DNG09, Gor10, GMN13, FK20, klein2020geometric, Ney}, schemes defined by certain toric ideals of graphs \cite{ConGor,CDRV}, and many classes of (or closely related to) monomial schemes \cite{MN00, MN02b, HU07, KTY13}.

Of particular relevance to this paper are the contributions made by Nagel and R\"omer \cite{NagelUwe2008Gsc} in the case of schemes defined by squarefree monomial ideals. 
Squarefree monomial ideals are associated to simplicial complexes through the Stanley--Reisner correspondence (see Section \ref{simplicialcomplexes}). Nagel and R\"omer introduced the class of weakly vertex decomposable simplicial complexes, which include matroid, shifted, Gorenstein, and vertex decomposable complexes.  They showed that the schemes corresponding to weakly vertex decomposable complexes are glicci. More specifically, they constructed a sequence of combinatorially defined \emph{basic double G-links} (see Definition \ref{def:double link}) from the original Stanley--Reisner subscheme to a coordinate subspace. 
Nagel and R\"omer also gave examples of naturally occurring complexes that are Cohen--Macaulay but not weakly vertex decomposable.  It is not known whether or not these examples are glicci.

In this paper, we extend the work of Nagel and R\"omer in two directions.  To state our first main result (which appears in its precise form as Theorem \ref{lem: linearform}), let $G$ be a graph, and let $I(G)$ be its corresponding \emph{edge ideal}. Call a homogeneous, saturated ideal of a polynomial ring in $n$ variables over a field glicci if the subscheme of $\mathbb{P}^{n-1}$ it defines is glicci. 

\begin{theorem*}
Let $I(G)$ be an edge ideal in the polynomial ring $S=\Bbbk[x_1,\dots, x_n]$. Assume there is a basic double G-link $I(G) = f B+A$ where $f$ is a homogeneous form in $S$,  $B\subset S$ is a monomial ideal, and $A\subset S$ is a homogeneous ideal. Then, up to rescaling the variables in $S$,
\begin{enumerate}
\item $f = x_{i_1}+\cdots + x_{i_r}$ is a sum of distinct indeterminates in $S$; and
    \item $B$ is a specific squarefree monomial ideal which is completely determined by $f$ and $G$. 
\end{enumerate}
\end{theorem*}

\noindent We use this result to show that the edge ideal corresponding to Figure \ref{fig:circGraph} is not a basic double G-link of any other monomial ideal.

Our second goal is to further explore Gorenstein liaison for Stanley-Reisner subschemes associated to specific simplicial complexes which are not weakly vertex decomposable. In Proposition \ref{prop:RP2}, we show that one of the examples studied in \cite{NagelUwe2008Gsc} is not a basic double G-link of any other Stanley--Reisner subscheme.  In Proposition \ref{prop:notRP2}, we consider another example from \cite{NagelUwe2008Gsc} and preclude the existence of a basic double G-link under the slightly stronger hypothesis that all of the ideals involved in the basic double G-link are squarefree monomial ideals.

If these examples are to be glicci, they either have to be G-linked to complete intersections using techniques other than basic double G-link or via a sequence of basic double G-links that includes schemes other than Stanley--Reisner schemes.  In Section \ref{sect:elementaryG}, we explore what might be gained by expanding beyond basic double G-links.

Throughout this document, let $\Bbbk$ be an arbitrary field and $S = \Bbbk[x_1, \ldots, x_n]$ be the standard graded polynomial ring over $\Bbbk$ in $n$ variables.  We will reserve $R$ for standard graded polynomial rings in specific examples, in which case $n$ will be known and $\Bbbk$ will be $\mathbb{Q}$.

\section{Background}\label{background}

In this section, we review some background material.  Section \ref{simplicialcomplexes} concerns simplicial complexes and Stanley-Reisner ideals, and Section \ref{liaison} covers some basic facts about  Gorenstein liaison.

\subsection{Simplicial Complexes}\label{simplicialcomplexes}

An abstract simplicial complex $\Delta$ on vertex set $[n] = \{1,\ldots, n\}$ is a collection of subsets of $[n]$, closed under inclusion. An element $F\in \Delta$ is called a face. The dimension of a face $F\in \Delta$ is defined by $\text{dim}(F) = \lvert F\rvert -1$, where $\lvert F\rvert$ denotes the number of elements in the subset $F\subseteq [n]$. The dimension of the simplicial complex $\Delta$ is defined by $\text{dim}(\Delta) = \text{max}\{\text{dim}(F)\mid F\in \Delta\}$.  A simplicial complex is called \textbf{pure} if each maximal face of $\Delta$ has the same dimension.   

 If $F\subseteq [n]$, define $x_F = \Pi_{i\in F}x_i$. To each simplicial complex there is an associated monomial ideal, called the Stanley-Reisner ideal of $\Delta$ and denoted $I_\Delta$.  The Stanley-Reisner ideal is generated by nonfaces of $\Delta$; that is $I_\Delta=(x_F: F\subseteq [n], F\not\in \Delta)$.  This correspondence provides a bijection between simplicial complexes $\Delta$ on the vertex set $[n]$ and squarefree monomial ideals $I_\Delta$ in $S$. We will refer to the elements of the (unique) minimal generating set of a monomial ideal that consists of monic monomials as the monomial generators of the ideal.  The monomial generators of the squarefree monomial ideal $I_\Delta$ correspond to the minimal nonfaces of $\Delta$.
 
 We note that $\text{dim}(S/I_{\Delta})$ is equal to $\text{dim}(\Delta)+1$ and that $\Delta$ is pure if and only if $I_\Delta$ is height unmixed, that is, if all of the associated primes of $I_\Delta$ have the same height.  Because $I_\Delta$ is a radical ideal and therefore has no embedded primes, $I_\Delta$ is height unmixed if and only if Spec$(S/I_\Delta)$ is equidimensional.  We say that a simplicial complex $\Delta$ is Cohen-Macaulay whenever $R/I_\Delta$ is Cohen--Macaulay.

Given a vertex $v$ of $\Delta$, we define the following subcomplexes of $\Delta$:
\begin{itemize}
\item the \textbf{link} of $v$,
$\operatorname{lk}_\Delta v = \{G\in \Delta \ \mid \{v\}\cup G\in \Delta, \ \{v\}\cap G=\emptyset\}$;
\item the \textbf{deletion} of $v$,
$\Delta_{-v} =\{ G \in \Delta \ \mid \{v\}\cap G=\emptyset\}$.
\end{itemize}

\noindent When we form the Stanley--Reisner ideals of $\operatorname{lk}_\Delta v$ and $\Delta_{-v}$, we view both of these complexes as complexes on the vertex set $[n] \setminus \{v\}$.

Given a simplicial complex $\Delta$ on $[n]$ and a vertex $k$ such that $\{k\} \notin \Delta$, define the \textbf{cone over $\Delta$ with apex $k$} to be the complex $\Delta' = \Delta \cup \{F \cup \{k\} | F \in \Delta\}$.  Notice that $I_\Delta = I_{\Delta'}+(x_k)$ and that $x_k$ does not divide any of the monomial generators of $I_{\Delta'}$.

A pure simplicial complex $\Delta$ is \textbf{vertex decomposable} if $\Delta$ is a simplex, if $\Delta=\{\emptyset\}$, 
 or if there exists a vertex $v$ such that $\operatorname{lk}_\Delta v$ and $\Delta_{-v}$ are both pure and vertex decomposable.  In the latter case, if $\dim(\operatorname{lk}_\Delta v)=\dim(\Delta_{-v})-1$, then we call $v$ a \textbf{shedding vertex}.  
 Vertex decomposition was introduced in \cite{PB80}.  We say that a pure simplicial complex $\Delta$ is \textbf{weakly vertex decomposable} if $\Delta$ is a simplex, if $\Delta = \{\emptyset\}$, or if there is some vertex $v$ such that lk$_{\Delta}v$ is weakly vertex decomposable and $\Delta_{-v}$ is Cohen-Macaulay.   In the latter case, if $\dim(\operatorname{lk}_\Delta v)=\dim(\Delta_{-v})-1$, then we call $v$ a \textbf{weak shedding vertex}.  Weak vertex decomposition was introduced in \cite{NagelUwe2008Gsc}. Whenever $\Delta$ is vertex decomposable, $\Delta$ is Cohen--Macaulay; hence, a vertex decomposable complex is weakly vertex decomposable, as the name suggests.  Nagel and R\"omer \cite[Theorem 3.3]{NagelUwe2008Gsc} showed that if $\Delta$ is even weakly vertex decomposable, then $\Delta$ is Cohen-Macaulay, and in fact $I_\Delta$ is \emph{glicci}  (defined below in Subsection \ref{liaison}).

\subsection{Gorenstein Liaison}\label{liaison}

Here we review some definitions and results from Gorenstein liaison. For more information, see the surveys \cite{migliore2002liaison, MN21}. All ideals in this subsection are assumed to be proper, homogeneous, and saturated.

Let $I,J,$ and $L$ be unmixed ideals of $S$, and assume that $S/L$ is Gorenstein.  We say that $I$ and $J$ are directly G-linked by  $L$, denoted $I\sim_L J$, if $L:I=J$ and $L:J=I$. Importantly, $S/I$ is Cohen--Macaulay if and only if $S/J$ is Cohen--Macaulay.  We say that $I$ and $J$ are in the same Gorenstein liaison class if there is a sequence of Gorenstein ideals $L_1,\dots, L_r$ and G-links $I=I_0\sim_{L_1} I_1\sim \cdots \sim_{L_r} I_r=J$, for some $r\geq 1$. If $J$ is a complete intersection, then we say that $I$ is in the Gorenstein liaison class of a complete intersection, abbreviated \textbf{glicci}. We say that a subscheme of $\mathbb{P}^{n-1}$ is glicci if its homogeneous, saturated ideal is glicci.  We call a simplicial complex $\Delta$ glicci if its Stanley--Reisner ideal $I_\Delta$ is glicci.  

If $(S/A)_P$ is Gorenstein for all minimal primes $P$ of $A$, we say that $S/A$ is \textbf{generically Gorenstein}. For example, if $A$ is radical, then $S/A$ is generically Gorenstein.

\begin{definition}\label{def:double link}
Let $A \subset B$ be unmixed ideals of $S$ such that $S/A$ is Cohen--Macaulay and generically Gorenstein and $\operatorname{ht}(A)+1=\operatorname{ht}(B)$. Let $f\in S$ be a homogeneous element of degree $d>0$ such that $A:f=A$. Then $C= fB+A$ is a \textbf{basic double G-link of degree $d$ of $B$ on $A$}. 
\end{definition}

 Though it is not standard, we will also refer to the equation $C = fB+A$ as a basic double G-link as well as the ideal $C$ itself.  

\begin{example}
Let $C=(x_1x_3,x_2x_4)\subset \Bbbk[x_1,x_2,x_3,x_4]$. Taking $f=x_1,$ $B=(x_3,x_2x_4)$ and $A=(x_2x_4)$, a basic double G-link is given by $(x_1x_3,x_2x_4)=x_1(x_3,x_2x_4)+(x_2x_4) $.
\end{example}

\begin{theorem}\cite[Proposition 1.3]{Memoir}
If $C$ is a basic double G-link of $B$ on $A$, then $C$ is G-linked to $B$ in two steps. 
\end{theorem}

Nagel and R\"omer \cite{NagelUwe2008Gsc} gave a connection between weak vertex decomposition and basic double G-linkage in the squarefree monomial ideal setting.  They considered the special case of basic double G-links in which $A$, $B$, and $C$ are all squarefree monomial ideals and $f = x_k$ is some variable of $S$.

Before stating their result, we make two standard observations: First, if $\Delta$ is pure and not a cone over $\Delta_{-k}$, then $\dim(\Delta) = \dim(\Delta_{-k})$.  Second, let $I$ be an ideal of $S$. If $I$ is glicci and $y$ is an indeterminate, then the extension of $I$ to the ring $S[y]$ is also glicci. Thus, if $\Delta$ is a glicci simplicial complex, then any cone over $\Delta$ is also glicci.  

\begin{lemma} \cite[Remark 2.4(iii)]{NagelUwe2008Gsc}\label{lem:wvd}
Let $\Delta$ be a simplicial complex with Stanley--Reisner ideal $I_\Delta$. If $\Delta$ is pure and if the deletion $\Delta_{-k}$ is Cohen--Macaulay and has the same dimension as $\Delta$, then $I_\Delta$ is a basic double G-link of the cone over the Stanley--Reisner ideal of its link $\operatorname{lk}_{\Delta}(k)$.  Conversely, if there exists a vertex $k$ of $\Delta$ so that $I_\Delta = x_k I_{\operatorname{lk}_{\Delta}(k)}+I_{\Delta_{-k}}$ is a basic double G-link, then $k$ is a weak shedding vertex of $\Delta$. 
\end{lemma}

In order to consider a wider variety of ways in which one squarefree monomial ideal might be a basic double G-link of another, we will work to drop the requirements that $A$ be a squarefree monomial ideal and that $f$ be a variable of $S$.

\begin{Remark}\label{rem:whyC}
As was the setting of \cite{NagelUwe2008Gsc}, the squarefree monomial ideals that are the focuses of our study will play the role of $C$ in the equation $C = fB+A$.  That is, we will be studying a squarefree monomial ideal $C$ and asking whether or not it might be a basic double G-link of another (typically squarefree) monomial ideal $B$.  The reason for this choice, as opposed to studying the ideal of interest in the role of $B$, is that the ideal $B$ is simpler than $C$ in the sense that every monomial in $B$ is a divisor of a monomial in $C$, and at least one monomial generator of $B$ is a proper divisor of some monomial generator of $C$. The goal is to gain an understanding of a more complicated ideal by relating it to a simpler ideal.  

 In light of the result in \cite{MN02} that every arithmetically Cohen--Macaulay subscheme of $\mathbb{P}^{n-1}$ embeds as a glicci subscheme of $\mathbb{P}^{n}$, we will work in a fixed polynomial ring in $n$ variables.  
 In particular, once we fix our ambient vertex set for our simplicial complexes, we will never expand this set when performing basic double G-links.
\end{Remark}

\section{basic double G-links of edge ideals}\label{edge}

In this section we study basic double G-links of \emph{edge ideals} (see Definition \ref{def:edge}).  Our main theorem (Theorem \ref{lem: linearform}) is a description of the form of a basic double G-link $I(G) = fB+A$ when $I(G)$ is an edge ideal and $B$ is any monomial ideal.  We then use this theorem to provide an example of an edge ideal that is Cohen--Macaulay but is not a basic double G-link of any other monomial ideal $B$ on any homogeneous ideal $A$.
This edge ideal was previously studied in \cite{IndComplex}.

We begin with a straightforward observation that is not specific to monomial ideals.

\begin{lemma}\label{lemdeg}
Let $A$, $B$, and $C = (c_1,\dots, c_r)$ be proper homogeneous ideals of $S$. Assume that each $c_i$ is a homogeneous polynomial, and set $d= \max_{i \in [r]} \text{deg}(c_i)$. If $f\in S$ is a homogeneous polynomial such that $C=fB+A$ is a basic double G-link of $B$ on $A$, then $1\leq \text{deg}(f)<d$. 
\end{lemma}

\begin{proof}
We proceed by contradiction and suppose that $\deg(f) \geq d$.  By definition of basic double G-link, $B \neq S$, and so every homogeneous element of $B$ has positive degree.  Hence, every homogeneous element of $fB$ has degree $>d$.   Since each $c_i$ has degree at most $d$, the condition $C = fB+A$ implies that $c_i\in A$ for all $1\leq i\leq r$.  Thus $C=A$, in violation of the condition $\hgt(A)+1 = \hgt(C)$.
\end{proof}

We now recall the definition of an edge ideal.

\begin{definition}\label{def:edge}
Let $G$ be a finite simple graph with vertex set  $V(G) = \{1,\dots, n\}$ and edge set $E(G)$. The \textbf{edge ideal} $I(G)$ of the graph $G$ is the squarefree monomial ideal of $S$ defined by 
\[
I(G)= (x_ix_j~:~ \{i,j\}\in E(G) ). \]
\end{definition}

We have the following immediate consequence of \Cref{lemdeg}:

\begin{corollary}\label{cor:deg}
Let $I(G)\subset S$ be an edge ideal, and suppose there exists a basic double G-link $I(G)=fB+A$ for some homogeneous polynomial $f\in S$. Then $\text{deg}(f) = 1$. \hfill \qedsymbol 
\end{corollary}

The next lemma will let us place further restrictions on the linear forms $f$ that we consider when thinking about basic double G-links of edge ideals. This lemma is not specific to monomial ideals. We omit the proof, which is straightforward.  

\begin{lemma}\label{lem:changeofvars}
Let $A, B, C\subset S$ be homogeneous ideals, and let $f$ be a nonzero linear form $f = \sum_{i=1}^n\alpha_i x_i$, where $\alpha_i\in \Bbbk$. Let $S' = \Bbbk[y_1,\dots y_n]$, and define a ring homomorphism $\phi:S'\rightarrow S$ given by $\phi(y_i) = \alpha_i x_i$ if $\alpha_i\neq 0$ and $\phi(y_i) = x_i$ otherwise. Let $A', B',$ and $C'$ be the kernels of the induced maps $S'\rightarrow S/A$, $S'\rightarrow S/B$, and $S'\rightarrow S/C$, respectively. Then $C = fB+A$ is a basic double G-link if and only if $C' = hB'+A'$ is a basic double G-link, where \[h = \sum_{\{j:\alpha_j\neq 0\}} y_j. \] 
\end{lemma}

Let $i\in V(G)$ be a vertex of the graph $G$. Define the set of \textbf{neighbours of $i$}, $\mathcal{N}(i)$, to be those vertices $j\in V(G)$ such that $\{i,j\}\in E(G)$. If $f=x_{i_1}+\cdots+ x_{i_r}$, define  $\mathcal{N}_f = \left(x_j ~:~ j\in \bigcap_{q=1}^r \mathcal{N}(i_q)\right)$.

\begin{theorem}
\label{lem: linearform}
Let $I(G)\subset S$ be an edge ideal. Assume there is a basic double G-link $I(G) = g B+\tilde{A}$, where $g$ is a homogeneous form in $S$, 
$B\subset S$ is a monomial ideal, and $\tilde{A}\subset S$ a homogeneous ideal. Then, there is a basic double G-link $I(G) = fB+A$, where $A$ is a homogeneous ideal and
\begin{enumerate}
\item $f = x_{i_1}+\cdots + x_{i_r}$ is a sum of distinct indeterminates in $S$; and
    \item $B =  I(G)+\mathcal{N}_f$. 
\end{enumerate}
\end{theorem}

\begin{proof}
Suppose $I(G) = gB+\tilde{A}$ is a basic double G-link. By Corollary \ref{cor:deg}, we have $\text{deg}(g) = 1$. Then, by Lemma \ref{lem:changeofvars} and the assumptions that $I(G)$ and $B$ are monomial ideals, we have a basic double G-link
\[
I(G) = fB+A
\]
for some homogeneous ideal $A$ and some sum of indeterminates $f = x_{i_1}+\cdots+ x_{i_r}$. 
It remains to prove (2).

We begin by proving that an indeterminate $z = x_i$ is a generator of $B$ if and only if it is one of the indeterminates in  $\mathcal{N}_f$. Indeed, if $z\in B$, then $zf\in I(G)$. Since $I(G)$ is a monomial ideal, we conclude that $zx_{i_q}\in I(G)$ for all $q\in [r]$, and hence $z\in \mathcal{N}_f$. Conversely, if $z\in \mathcal{N}_f$ then $zx_{i_q}\in I(G)$ for each $q\in [r]$, and so $fz\in I(G)$. 
As $I(G) = fB+A$, we have $fz = fb + a$ for some $b\in B$ and some $a\in A$. Thus, $f(z-b) = a\in A$. Since $f$ is a nonzerodivisor on $S/A$, we have that $z-b \in A$. As $A\subset B$ by the definition of basic double G-link, we have $z \in B$ as desired.

We next show that $I(G)+\mathcal{N}_f\subseteq B$. We have already seen that $\mathcal{N}_f\subseteq B$, so it suffices to prove that $I(G)\subseteq B$. 
Fix a minimal monomial generator $c\in I(G)$. 
Because $I(G)$ is an edge ideal, $\deg(c) = 2$. 
If $c$ is a multiple of some $y\in \mathcal{N}_f$, then $c \in B$.
If $c$ is not a multiple of any indeterminate $y\in \mathcal{N}_f$, then $c$ is not a term of any element of $fB$; if it were, then $c$ would be a term of some $fw$, where $w$ is an indeterminate in $B$. But, as shown above, if $w\in B$, then $w\in \mathcal{N}_f$.
It follows that $c$ is a term of some element $g\in A$. Since $A\subset B$, we have $g\in B$. As $B$ is a monomial ideal, it follows that $c\in B$.

Finally, we show that $B\subseteq I(G)+\mathcal{N}_f$. 
Fix some monomial $b \in B$. 
If $b\in \mathcal{N}_f$, then we are done. 
So suppose otherwise. 
We have $fb\in I(G)$, and so $x_{i_q}b\in I(G)$ for all $q\in [r]$ since $I(G)$ is a monomial ideal. 
Let $c_1,\dots, c_r$ be degree $2$ monomial generators of $I(G)$ such that $c_q$ divides $x_{i_q}b$ for each $q\in [r]$. If $c_q$ divides $b$ then $b\in I(G)$ as desired. So we may assume that, for each $q \in [r]$, there is an indeterminate $z_q$ such that $c_q = x_{i_q}z_q$.

Since $c_q\in I(G)$ and $I(G) = fB+A$, and $A$ and $B$ are  homogeneous ideals, we have $c_q = fb_q+a_q$ for some $b_q\in B$ of degree $1$ and some $a_q\in A$ of degree $2$. 
Write $x_{i_q}b = c_qm_q$. 
Then, \[x_{i_q}b = fm_qb_q+m_qa_q,\]
and so adding up these equalities yields
\[
fb = x_{i_1}b+\cdots + x_{i_r}b = f(m_1b_1+\cdots m_rb_r)+(m_1a_1+\cdots +m_ra_r).
\]
So, $f(b-m_1b_1-m_2b_2-\cdots - m_rb_r)\in A$. Since $f$ is not a zero divisor of $S/A$, we have that
$b-m_1b_1-m_2b_2-\cdots - m_rb_r\in A$. As $A\subset I(G)$, we have that $b-m_1b_1-m_2b_2-\cdots - m_rb_r\in I(G)$. Since $I(G)$ is a monomial ideal, it will follow that  $b\in I(G)$ provided that the coefficient of the monomial $b$ in $b-m_1b_1-m_2b_2-\cdots - m_rb_r$ is nonzero. To see this, note that no element of $\mathcal{N}_f$ divides $b$ by assumption. On the other hand, since each $b_q$ is a degree $1$ element of $B$, it follows that every term of $m_1b_1+\cdots+m_rb_r$ is divisible by some element of $\mathcal{N}_f$. Thus, no term of $m_1b_1+\cdots+m_rb_r$ is equal to $b$.
\end{proof}

\Cref{lem: linearform} leads to the following corollary regarding the maximum length of a linear form that can be used in a basic double G-link.
Recall that the \textbf{degree} of a vertex $v\in V(G)$, which we denote by $\text{deg}(v)$, is the number of edges in $E(G)$ which have $v$ as an endpoint.

\begin{corollary}\label{cor: maxCommonNeigh}
Let $I(G)$ be an edge ideal of $S$, $B\subset S$ be a monomial ideal, and $A\subset S$ be an arbitrary homogeneous ideal. Assume that $f=x_{i_1}+\ldots +x_{i_r}$ is a sum of indeterminates and that $I(G) = fB+A$ is a basic double G-link. Then $\text{max}\{\text{deg}(v)\mid v\in V(G)\}\geq r$.    
\end{corollary}

\begin{proof}

If $\mathcal{N}_f = \emptyset$, then $B = I(G)$ by Theorem \ref{lem: linearform}. Thus, $fB$ contains no element of degree $2$, and we conclude from the basic double G-link $I(G) = fB+A$ that $A = I(G)$, in violation of the condition $\hgt(I(G)) = \hgt(A)+1$. Thus, $\mathcal{N}_f\neq \emptyset$.

So, let $x_i \in \mathcal{N}_f$. Then $i$ is a neighbour of $i_q$ for all $q\in [r]$. Thus, $\text{deg}(x_i)\geq r$, and it follows that $r\leq \text{max}\{\text{deg}(v)\mid v\in V(G)\}$.
\end{proof}

Before we use Theorem \ref{lem: linearform} 
in a specific example, we consider a possible converse 
to Theorem \ref{lem: linearform}.  
In Theorem \ref{lem: linearform}, we began from the information of a basic double G-link and constrained precisely (up to scaling the coordinates) the form $f$ and the ideal $B$.  The following example shows that, if one starts with a suitable form $f$ and its corresponding ideal $B$ as in  Theorem \ref{lem: linearform} (2), one still need not have a basic double G-link.  That is, even if one chooses an edge ideal $I(G)$, a Cohen--Macaulay, generically Gorenstein, homogeneous ideal $A$ of height one less than the height of $I(G)$, and a sum of indeterminates $f$ that is not a zerodivisor on $S/A$, it is not guaranteed that $I(G) = fB+A$, where $B = I(G)+\mathcal{N}_f$.

\begin{example}
Let $I(G) = (x_1x_2,x_2x_3, x_3x_4)$, $A = (x_3x_4)$, and $f = x_1$.  Then $\mathcal{N}_f = (x_2)$, and so, if $B = I(G)+\mathcal{N}_f = (x_3x_4,x_2)$, then $I(G) \neq fB+A = (x_1x_2,x_3x_4)$.  
\end{example}

We end this section by considering the edge ideal of the circulant graph $G = C_{16}(1,4,8)$ (see Figure \ref{fig:circGraph}).  The corresponding simplicial complex has dimension $3$ and is not vertex decomposable \cite[Theorem 6.1]{IndComplex}.  Indeed, the proof of \cite[Theorem 6.1]{IndComplex} shows that the corresponding simplicial complex is also not weakly vertex decomposable.

\begin{example}
\label{ex:circGraph}
Consider the circulant graph $G=C_{16}(1,4,8)$. Its edge ideal $I(G) \subset R=\mathbb{Q}[x_1,\ldots,x_{16}]$, which was studied in \cite{IndComplex}, is given by
\begin{align*}
I(G)=&(x_1x_2,x_2x_3,x_3x_4,x_4x_5,x_5x_6,x_6x_7,x_7x_8,x_8x_9,x_9x_{10},x_{10}x_{11},x_{11}x_{12},x_{12}x_{13},x_{13}x_{14},x_{14}x_{15},\\
& x_{15}x_{16},x_1x_{16},x_1x_9,x_1x_5,x_1x_{13},x_2x_6,x_2x_{14},x_2x_{10},x_3x_{7},x_3x_{11},x_3x_{15},x_4x_8,x_4x_{12},x_4x_{16},\\
& x_5x_9,x_5x_{13},x_6x_{10},x_6x_{14},x_7x_{11},x_7x_{15},x_8x_{12},x_8x_{16},x_9x_{13},x_{10}x_{14},x_{11}x_{15},x_{12}x_{16}). 
\end{align*}
\end{example}

We will next use the results of this section to show that there is no basic double G-link $I(G) = fB+A$ for any choice of form $f$, homogeneous ideal $A$, and monomial ideal $B$. We need one more lemma, which follows easily from the definition of basic double G-link.

\begin{figure}[h]
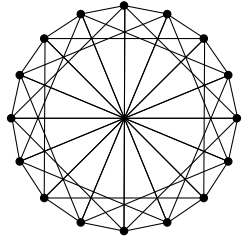

\[\Circulant{16}{1,4,8}\]
\caption{The circulant graph $C_{16}(1,4,8)$. Note that the $16$ vertices are equally spaced around the outer boundary of the figure; there is no vertex in the centre.}
\label{fig:circGraph}
\end{figure}

\begin{lemma}\label{lem:fisxi}
Let $C$ be a homogeneous, saturated, unmixed ideal of $S$.  If there exists a basic double G-link of the form $C = x_iB+A$, then $S/(C+(x_i))$ is Cohen--Macaulay.
\end{lemma}

\begin{proof}
Since $C = x_iB+A$ is a basic double G-link, we have that $S/A$ is Cohen-Macaulay and $A:x_i = A$. Thus, $S/(A+(x_i))$ is Cohen--Macaulay.
Furthermore, $C+(x_i) = x_iB+A+(x_i) = A+(x_i)$. Therefore, $S/(C+(x_i))$ is Cohen--Macaulay.
\end{proof}

\begin{proposition}
For $I(G)$ defined in Example \ref{ex:circGraph}, there is no basic double G-link $I(G) = fB+A$ for any monomial ideal $B$.
\end{proposition}
\begin{proof}
Suppose there exists a basic double G-link $I(G)=fB+A$, where $f$ is a homogeneous form, $A$ is a homogeneous ideal, and $B$ is a monomial ideal. 
By Theorem \ref{lem: linearform} and Lemma \ref{lem:changeofvars}, we may assume that $f= x_{i_1}+\cdots +x_{i_k}$ for some  $x_{i_j}\in \{x_1,\ldots, x_{16}\}$ and $B = I(G)+\mathcal{N}_f$. 

By the symmetry of $G$, we may assume that $x_{i_1} = x_1$. Thus, $\mathcal{N}_f$ is generated by a subset of the set 
\[X = \{x_j\mid i\in \mathcal{N}(1)\} =  \{x_2, x_5, x_9, x_{13}, x_{16}\}.\]
Since $R/B$ must be Cohen-Macaulay, one checks in Macaulay2 \cite{M2} that the only candidates for $B$ are: $J_1 = I(G)+(x_2,x_9,x_{16})$, $J_2 = I(G)+(x_2,x_5,x_9,x_{16})$, $J_3 = I(G)+(x_2,x_9,x_{13},x_{16})$, and $J_4 = I(G)+(x_2,x_5,x_9,x_{13},x_{16})$. 
However, the only vertex which is a neighbour to both $2$ and $16$ is vertex $1$. This means that none of $J_1, J_2, J_3$ are of the form $I(G)+\mathcal{N}_f$ for some $f$. So, it remains to rule out the existence of a basic double G-link $I(G) = fB+A$, where $B = I(G)+(x_2,x_5,x_9,x_{13},x_{16})$ and $f = x_1$.  One may check in Macaulay2 that $R/(I(G)+(x_1))$ is not Cohen--Macaulay.  Hence, by Lemma \ref{lem:fisxi}, $I(G) = fB+A$ is not a basic double G-link.
\end{proof}

\section{basic double G-links and other squarefree monomial ideals}\label{general}

In this section, we expand beyond the class of edge ideals to consider other Cohen--Macaulay Stanley--Reisner ideals $I_\Delta$ whose associated simplicial complexes are not weakly vertex decomposable.  

In particular, in the examples we consider, $I_\Delta$ will be known not to admit any basic double G-link of the form $I_\Delta = x_k I_{\operatorname{lk}_{\Delta} k}+I_{\Delta_{-k}}$.  In the first example, we will show that the ideal in question is not a basic double G-link of any squarefree monomial ideal on any homogeneous ideal. In the second example, we rule out only the case of a basic double G-link of any squarefree monomial ideal on any other squarefree monomial ideal.  We use the second example, which looks superficially quite similar to the first, to highlight how subtle the issue of assessing the possibility of a basic double G-link is.  

This section involved many computer computations, all of which were performed in Macaulay2 \cite{M2}.

\begin{example}\label{RP2}
Consider the ideal  

\[
I_\Delta=(x_1x_2x_3, x_1x_2x_4, x_1x_3x_5, x_1x_4x_6, x_1x_5x_6, x_2x_3x_6, x_2x_4x_5, x_2x_5x_6, x_3x_4x_5, x_3x_4x_6).
\] 
\end{example}

The ideal $I_\Delta$ corresponds to a triangulation of $\mathbb{R}\mathbb{P}^2$ and was studied as \cite[Example 5.2]{NagelUwe2008Gsc} as an example of an ideal whose associated simplicial complex is Cohen--Macaulay but not weakly vertex decomposable.  The height of $I_\Delta$ is $3$.

\begin{proposition}\label{prop:RP2}
For $I_\Delta$ defined in Example \ref{RP2}, as an ideal of the ring $R = \mathbb{Q}[x_1, \ldots, x_6]$, there is no basic double G-link of the form $I_\Delta = fB+A$ for any squarefree monomial ideal $B$.
\end{proposition}

\begin{proof}
Our approach in this proof is to suppose that we have a basic double G-link $I_\Delta=fB+A$ and find restrictions on $f$, $B$, and $A$. We will use these restrictions find many polynomials that must be in the ideal $A$ with the goal of concluding that the height of $A$ is at least $3$.  Because $\hgt(I_\Delta)=3$, by the definition of basic double G-link, $A$ is required to have height $2$.  Hence, $\hgt(A) \geq 3$ will constitute a contradiction.

It follows from Lemma \ref{lemdeg} that, if $I_\Delta$ is a basic double G-link of any homogeneous ideal $B$, $f$ must be of degree at most $2$.  Suppose first that $\deg(f) = 1$.  Consider first the case of $f = x_i$ for some $i \in [6]$.  The fact that $R/(I_\Delta+(x_i))$ is not Cohen--Macaulay for any $i \in [6]$ is computed in \cite[Example 5.2]{NagelUwe2008Gsc}. Hence, by Lemma \ref{lem:fisxi}, no basic double G-link of the form $I_\Delta = fB+A$ exists.

Suppose $f=q_i x_{i_1}+\cdots+ q_n x_{i_n}$ for some $2\leq n\leq 6$ and constants $0 \neq q_i \in \mathbb{Q}$. By Lemma \ref{lem:changeofvars}, we may assume $q_i = 1$ for all $i$.  If all generators of $B$ have degree greater than $2$, then  every homogeneous polynomial of $fB$ has degree greater than $3$.  Thus, the equality $I_\Delta=fB+A,$ together with the fact that $I_\Delta$ is generated in degree $3$, would imply $A=I_\Delta$, a contradiction. Also, because $I_\Delta$ is generated in degree $3$ and $fB \subseteq I_\Delta$, no form of degree strictly less than $2$ can be an element of $B$.  

Hence, there must be at least one generator $z\in B$ of degree exactly $2$.  Write $z = z_1z_2$, where $z_1, z_2 \in \{x_1, \ldots, x_6\}$ and $z_1 \neq z_2$.  Because $I_\Delta$ is a monomial ideal and $fz \in I_\Delta$, it must be that $x_{i_j}z_k \in I_\Delta$ for all $j \in [n]$ and $k \in [2]$.  A computer computation shows that, for any choice of $z_1$ and $z_2$, there are at most two choices of $j$ satisfying $x_{i_j}z_k \in I_\Delta$ for $k = 1,2$.  Hence, we take $n = 2$ and write $f = x_{i_1}+x_{i_2}$.  

It is now easy to check computationally that, for each choice of $f = x_{i_1}+x_{i_2}$, there is exactly one $z = z_1z_2$ satisfying $x_{i_j}z \in I_\Delta$ for $j = 1,2$.  Hence, $B = (z)+B'$ for some $B'$ generated in degrees $3$ and higher, in which case $fB'$ is generated in degrees $4$ and higher.  For each monomial generator $\mu$ of $I_\Delta$ not divisible by $z$, there must be an equation of the form $\mu = fb+a$ for some $b \in B$ and $a \in A$.  By homogeneity, we may assume that $\deg(fb) = 3 = \deg(a)$.  Thus, there is an equation of the form $\mu = \alpha_\mu fz+a$ for some $a \in A$ and $\alpha_\mu \in \{0,1\}$, and so $\mu - \alpha_\mu fz \in A$ for each such $\mu$. 

We now claim that exactly one of $x_{i_1}z$ or $x_{i_2}z$ is an element of $A$.  Indeed, if $x_{i_1}z \notin A$, then, because $x_{i_1}z \in I_\Delta = fB+A$, we must have an expression of the form $x_{i_1}z = fz-a$ with $a \in A$.  We now solve $x_{i_2}z = a \in A$.  If $x_{i_1}z \in A$ and $x_{i_2}z \in A$, then $fz \in A$.  Because $\deg(z) = 2$ and $A \subset I_\Delta$, we know $z \notin A$. Hence, $z \in A:f \setminus A$, in violation of the definition of basic double G-link.

For each choice of $z$, each choice of $x_{i_j}z \in A$, and each choice of $\alpha_\mu = 0$ or $\alpha_\mu = 1$ for each monomial generator $\mu$ of $I_\Delta$ not divisible by $z$, let $A' \subseteq A$ be the ideal generated by the $\mu-\alpha_\mu fz$ together with $x_{i_j}z$.  In all cases, $A'$ has height $3$.  Hence, $\hgt(A) \geq 3$, in violation of the definition of basic double G-link.  Thus, there can be no basic double G-link $I_\Delta = fB+A$ with $B$ squarefree monomial and $f$ a form of degree $1$.

Next, suppose that $f=q_1y_1+\cdots +q_n y_n$ where $n \geq 1$, $\deg(y_i)=2$ for all $i\in [n]$, and $0 \neq q_i \in \mathbb{Q}$.  Then, by an argument similar to the degree $1$ case above, $B$ must contain a monomial generator $w \in \{x_1,\ldots,x_6\}$ such that $y_iw$ is a monomial generator of $I_\Delta$ for each $i \in [n]$.  Because $y_iw \in I_\Delta$, $\deg(y_iw) = 3$, and $I_\Delta$ is a squarefree monomial ideal generated in degree $3$, each $y_iw$ must be squarefree.

Suppose $n=1$, and fix $f = q_1 x_{i_1}x_{i_2}$.  Then a computer computation show that the restriction $fw \in I_\Delta$ forces there to be exactly $0$ or $2$ allowable choices of $w$.  Exclude the case of $0$, which does not give rise to a basic double G-link, and call the $2$ allowable choices $w_1$ and $w_2$. That is, at least one of $w_1$ and $w_2$ is an element of $B$, and every degree $1$ homogeneous element of $B$ is in the ideal $(w_1,w_2)$.

For each choice of $f$, one may check computationally that there exists at least one monomial generator $\mu$ of $I_\Delta$ that is not an element of $(fw_1, fw_2)$ but is divisible by at least one of $x_{i_1}$ or $x_{i_2}$.  Without loss of generality, assume $x_{i_1}$ divides $\mu$. 

Because all monomial generators of $I_\Delta$ are of degree $3$ and every degree $3$ homogeneous polynomial in $fB$ is in the ideal $(fw_1, fw_2)$, the equality $I_\Delta = fB+A$ implies the existence of an equality of the form $\mu = \alpha_1 fw_1+\alpha_2 fw_2+a$ for some $a \in A$ of degree $3$ and scalars $\alpha_i$.  But then, because $x_{i_1}$ divides $\mu$ and $f$, it must also divide $a$.  Because $A \subset I_\Delta$ and $I_\Delta$ is generated in degree $3$, we have that $a/x_{i_1} \notin A$.   But $(a/x_{i_1})f = a(q_1x_{i_2}) \in A$, in violation of $A:f = A$.  Hence, we cannot have $n = 1$.

 If $n \geq 2$, then, for each choice of $f$ there is at most one choice of degree $1$ monic monomial $w$ satisfying $wy_i \in I_\Delta$ for all $i \in [n]$.  Hence, we may consider $6$ symmetric cases $w = x_j$, for a some $j \in [6]$.   For a fixed choice of $j$, consider the possible choices of $f$.  For each $j$, there are $10$ possible $f$ with $n=2$, ten  with $n=3$, $5$ with $n=4$, $1$ with $n=5$, and none with $n \geq 6$.  We will now work to reduce to the case of all but at most one of the coefficients in $f$ equal to $1$ for all for $2 \leq n \leq 5$.  
 
 If $n = 2$, in which case $f = q_1y_1+q_2y_2$, then there is at least one variable $x_{i_1} \neq w$ dividing $y_1$ but not $y_2$ and at least one variable $x_{i_2} \neq w$ dividing $y_2$ but not $y_1$. By Lemma \ref{lem:changeofvars}, we may apply the change of variables using  $x_{i_1} \mapsto q_1^{-1}x_1$ and $x_{i_2} \mapsto q_2^{-1}x_{i_2}$ to assume that $q_1 = q_2 = 1$. 
 
 If $n = 3$, the argument is only slightly more complicated: after possibly reordering the $y_i$, for each possible $f$, one can always choose a variable $x_{i_1} \neq w$ dividing $y_1$, a variable $x_{i_2} \neq x_{i_1}, w$ dividing $y_2$, and a variable $x_{i_3} \neq x_{i_1}, x_{i_2}, z$ dividing $y_3$ but not dividing either $y_1$ or $y_2$.  There is then always a change of variables sending each $x_i$ to a nonzero $\mathbb{Q}$-multiple of $x_i$ so that the coefficients of the terms of the image of $f$ are all $1$.  For example, if $f = q_1x_2x_3+q_2x_3x_5+q_3x_5x_6$ and $w = x_1$, then a satisfying change of variables is $x_2 \mapsto q_1^{-1}q_2x_2$, $x_5 \mapsto q_2^{-1}x_5$, and $x_6 \mapsto q_2q_3^{-1}x_6$.  The case $n=4$ is the same argument.  
 
In each of these cases $n = 2, 3, 4$, for each monomial generator $\mu$ of $I_\Delta$, $A$ must contain an element $a_\mu$ satisfying $\mu = \alpha_\mu fw+a_\mu$ for $\alpha_\mu \in \{0,1\}$. For all such choices, $\hgt(A) \geq  3 = \hgt(I_\Delta)$, in violation of the definition of basic double G-link, following the argument from the $\deg(f) = 1$ case.  

Finally, suppose $n = 5$.  In this case, we may not be ably to apply a change of variables as in the cases $n = 2,3,4$ to assume that $q_i = 1$ for all $i$.  For example, we may have that $w = x_1$ and $f = q_1 x_2x_3+q_2x_2x_4+q_3x_3x_5+q_4x_4x_6+q_5x_5x_6$.  Up to a choice of $q_i$, there is one such possibility for each $w = x_i$, $i \in [6]$, which is obtained by some permutation of the indices of the $x_i$ in the form $f$ in the case $w = x_1$ above. Assume without loss of generality that $w = x_1$.

Then the change of variables \[
\hfill x_2 \mapsto q_1^{-1}x_2, \hspace{1cm} x_4 \mapsto q_1q_2^{-1}x_4, \hspace{1cm} x_6 \mapsto q_2q_4^{-1}x_6, \hspace{1cm} x_5 \mapsto q_4q_5^{-1}x_5 \hfill
\]
 allows us to reduce to the case of $f = x_2x_3+x_2x_4+qx_3x_5+x_4x_6+x_5x_6$ for $q = r/s$, with $r, s \in \mathbb{Z}\setminus \{0\}$ and gcd$(r,s) = 1$.  Then $sf \in \mathbb{Z}[x_1, \ldots, x_6]$.  Let $\overline{g} = \overline{sf} \in \mathbb{Z}/2[x_1, \ldots, x_6]$, where $g = x_3x_5$ if $s$ is even (in which case $r$ must be odd) and $g = x_2x_3+x_2x_4+x_3x_5+x_4x_6+x_5x_6$ or $g = x_2x_3+x_2x_4+x_4x_6+x_5x_6$, depending on the parity of $r$, if $s$ is odd.  

Let $\widehat{A} = A \cap \mathbb{Z}[x_1, \ldots, x_6]$ in $S = \mathbb{Z}[x_1, \ldots, x_6]$ and $\overline{A}$ be the image of $\widehat{A}$ in $T = \mathbb{Z}/2[x_1, \ldots, x_6]$. We will show that $\hgt(A) \geq 3$, in violation of the definition of basic double G-link, by showing that $\hgt(A) = \hgt(\overline{A})$ and then that $\hgt(\overline{A}) \geq 3$.

We will show first that $\hgt(A) = \hgt(\overline{A})$. Let $m$ be the ideal of $S$ generated by $\{x_1,\dots, x_6\}$ and $mR$ the expansion of $m$ to $R$.  Notice that $S_m \cong R_{mR}$ and that $AR_{mR} = \widehat{A}R_{mR}$.  Because $A \subseteq mR$ and $\widehat{A} \subseteq m$, $\hgt(A) = \hgt(AR_{mR}) = \hgt(\widehat{A}R_{mR}) = \hgt(\widehat{A})$.  

Because $2$ is not a zerodivisor on $S/\widehat{A}$, we know $\hgt(\widehat{A}) = \hgt(\widehat{A}+(2))-1$.  Then because $T/\overline{A} \cong S/(\widehat{A}+(2))$ while $\dim(T) = \dim(S)-1$, we have $\hgt(\overline{A}) = \hgt(\widehat{A}+(2))-1$.  Combining these equations, $\hgt(A) = \hgt(\widehat{A})= \hgt(\overline{A})$.

We will next show that $\hgt(\overline{A}) \geq 3$.  Let $\overline{B}$ and $\overline{I_\Delta}$ be the ideals of $\mathbb{Z}/2[x_1, \ldots, x_6]$ generated by the monomial generators of $B$ and $I_\Delta$, respectively, and let $\overline{g}$ and $\overline{w}$ denote the images of $g$ and $w$, respectively.  Note that $\overline{I_\Delta} = \overline{g}\overline{B}+\overline{A}$.  Hence, for each monomial generator $\mu$ of $I_\Delta$, there must be an element $a_\mu \in \overline{A}$ such that $\mu = \alpha_\mu \overline{gw}+a_\mu$ with $\alpha_\mu \in \{0,1\}$. By the same argument use in the preceding cases, $\overline{A}$ must also contain all monomial generators of $\overline{I_\Delta}$ not divisible by $w$.  For all choices of $g$ and all choices of $\alpha_\mu$, a computer check establishes $\hgt(\overline{A}) \geq 3$.    

Hence, there cannot be a basic double G-link in which $\deg(f) = 2$ either, and so there is no basic double G-link $I_\Delta = fB+A$ for any squarefree monomial ideal $B$.
\end{proof}

 Proposition \ref{prop:RP2} is proved by contradiction.  With notation from Proposition \ref{prop:RP2}, we supposed that there was a basic double G-link of the form $I_\Delta = fB+A$ for a squarefree monomial ideal $B$ and considered several cases, arriving at a contradiction in each case.  However, we did not need the full strength of the assumption that the equality $I_\Delta = fB+A$ constituted a basic double G-link in order to arrive at contradiction in any of the cases.  Specifically, we did not use the supposition that $R/A$ was required to be Cohen--Macaulay because we were able to show that there was no equation of the form $I_\Delta = fB+A$ for $A:f= A$, $\operatorname{ht}(A)+1=\operatorname{ht}(B)$, and $B$ monomial.  In the next example, an equation satisfying the height requirements can be found, and substantial care is required to make use of the assumption of Cohen--Macaulayness of $R/A$ in order to arrive at a contradiction.

Though the following lemma is not directly in \cite{NagelUwe2008Gsc}, it is very similar to content covered by their Remark 2.4.  We include it for completeness.

\begin{lemma}\label{lem:MonBDLImpWVD}
If $A$, $B$, and $I_\Delta$ are squarefree monomial ideals, $I_\Delta$ is the Stanley--Reisner ideal of the complex $\Delta$, and $I_\Delta = x_kB+A$ is a basic double G-link, then $B$ is the Stanley--Reisner ideal of $\operatorname{lk}_\Delta k$ and $A$ is the Stanley--Reisner ideal of the cone over $\Delta_{-k}$ with apex $k$.  
\end{lemma}

\begin{proof}
Let $\Delta_A$ denote the Stanley--Reisner complex of $A$ and $\Delta_B$ the Stanley--Reisner complex of $B$.

Because $I_\Delta = x_kB+A$ is a basic double G-link, we know $A:x_k = A$, which is to say that $x_k$ does not divide any of the monomial generators of $A$.  Hence, in order to show that $A$ is the ideal of the cone over $\Delta_{-k}$ with apex $k$, it suffices to show that $A+(x_k)$ is the Stanley--Reisner ideal of $\Delta_{-k}$.  Notice that $A+(x_k) = I_\Delta+(x_k)$ and that $I_\Delta+(x_k)$ is the Stanley--Reisner ideal of $\Delta_{-k}$.

For a subset $F$ of $[n]$ satisfying $k \notin F$, we have $F \in \Delta_B$ if and only if $x_F \notin B$ if and only if $x_k x_F \notin I_\Delta$ if and only if $F \cup \{k\} \in \Delta$.  For a subset $F$ of $[n]$ with $k \in F$, then $F \notin \Delta_{-k}$, and we claim $x_F \in B$ is not a minimal generator of $B$.  If $x_F \in B$, then $x_k x_F \in I_\Delta$.  Because $I_\Delta$ is a squarefree monomial ideal and $x_k \mid x_F$, where $x_F \in I_\Delta$.  Because $x_k \mid x_F$ and $A:x_k = A$, we have $x_F \in x_kB$, in which case $x_{F\setminus \{k\}} \in B$, and so $x_F$ is not a minimal generator of $B$.  Hence, $B$ is the Stanley--Reisner ideal of $\operatorname{lk}_\Delta k$, as desired.
\end{proof}

\begin{example}\label{ex:notRP2} 
Let 
\[
\Delta=\{x_1x_5x_6,x_2x_4x_6,x_1x_4x_6,x_2x_3x_6,x_1x_2x_6,x_3x_4x_5,x_1x_4x_5,x_2x_3x_5,x_1x_3x_5,x_2x_3x_4\} 
\]

\noindent and $I_\Delta$ be the Stanley-Reisner ideal of $\Delta$, in which case

\[ I_\Delta=(x_2x_3x_4,x_1x_3x_5,x_2x_3x_5,x_1x_4x_5,x_3x_4x_5,x_1x_2x_6,x_2x_3x_6,x_1x_4x_6,x_2x_4x_6,x_1x_5x_6).
\]

\end{example}

It was shown in \cite[V6F10-6]{MoriyamaSonoko2003Icpi} that $\Delta$ is a shellable (hence Cohen--Macaulay) but not vertex decomposable simplicial complex and, in \cite[Example 5.4]{NagelUwe2008Gsc}, that it is moreover not weakly vertex decomposable. The height of $I_\Delta$ is $3$. 

\begin{proposition}\label{prop:notRP2}
Let $I_\Delta \subseteq \mathbb{Q}[x_1, \ldots, x_6]$ be as in Example \ref{ex:notRP2}.  There does not exist a basic double G-link of the form $I_\Delta = A+fB$ for squarefree monomial ideals $A$ and $B$ and a form $f$ that is a sum of monic monomials.
\end{proposition}
\begin{proof}

We know from \Cref{lemdeg} that, if there is a basic double G-link of the form $I_\Delta = fB+A$, then $\deg(f)<3$.  We will consider two broad cases, each of which has several subcases.  

\emph{Case} 1. is that $\deg(f) = 1$.  It was checked in \cite[Example 5.4]{NagelUwe2008Gsc} that $\Delta$ is not weakly vertex decomposable. Therefore, by Lemma \ref{lem:MonBDLImpWVD}, there does not exist a basic double G-link of the form $I_\Delta = x_kB+A$ for any $k \in [6]$.  From this fact, if there is a basic double G-link of the form $I_\Delta = fB+A$ where $f$ is a form of degree $1$ that is a sum of monic monomials, then $f = x_{i_1}+\cdots+x_{i_n}$ for some $n \geq 2$.  Each monomial generator $\mu$ of $I_\Delta$ must either be an element of $A$ or satisfy an equation of the form $\mu = b_\mu f+a_\mu$ for some degree two element $b_\mu \in B$ and some $a_\mu \in A$. 

The ideal $B$ must have a positive number of degree $2$ monomial generators (or else $I_\Delta = A$).  Whenever $z \in B$ is a degree $2$ monomial, we claim that $zx_{i_j} \in A$ for all but exactly one choice of $j \in [n]$.  If $zx_{i_j} \in A$ for all $j \in [n]$, then $z \in (A:f) \setminus A$, in violation of the definition of basic double G-link.  To see that we cannot have $zx_{i_j} \notin A$ for two or more values of $j$, note first that $zf \in I_\Delta$ implies $zx_{i_j} \in I_\Delta$ for all $j \in [n]$ because $I_\Delta$ is a monomial ideal, and so we must have an equality $zx_{i_j} = fz-a$ for some $a \in A$.  Because $a=fz-zx_{i_j}$ is a sum of all but one of the monomials of $fz$ and $A$ is a monomial idea, $A$ contains for all but one choice of $zx_{i_j}$, as claimed.  Note also that $A$ must contain all degree $3$ monomial generators of $I_\Delta$ not divisible by any of the degree $2$ monomials of $B$.

\emph{Subcase} $1$\emph{a}.  Assume that $B$ has exactly one degree $2$ monomial generator $z$.  From the discussion above, $A$ must contain all but exactly one of the generators of $I_\Delta$.  A computer check shows that there are five height $2$ squarefree monomial ideals contained in $I_\Delta$ containing all but exactly one of the generators of $I_\Delta$ and that none of these ideals is Cohen--Macaulay.

\emph{Subcase} $1$\emph{b}.  We assume that $B$ contains two or more degree $2$ monomials.  For each degree $2$ monomial $z \in B$ and each summand $x_i$ of $f$, we must have $zx_i \in I_\Delta$ because $I_\Delta$ is a monomial ideal.  A computer check shows that, for any set $S$ of two or more squarefree monomials of degree $2$, there is at most one $i$ so that $zx_i \in I_\Delta$ for each degree $2$ generator $z$ of $S$.  Hence, subcase $1b$ reduces to the case of $f = x_k$ treated above via \cite[Example 5.4]{NagelUwe2008Gsc} and Lemma \ref{lem:MonBDLImpWVD}.

\emph{Case} $2$. Assume that $\deg(f) = 2$.  In this case, write \[
f =z_1+\cdots+ z_n
\] for some $n \geq 1$ where each $z_i$ is a monic monomial of degree $2$.  Our next goal is to reduce to the case of $n \leq 2$ using our previous work.  

Suppose $n = 3$, and write $f = z_1+z_2+z_3$.  For each choice of $f = z_1+z_2+z_3$, let $\mathcal{U}_f$ be the set of subsets $U$ of $\{x_1,\ldots, x_6\}$ satisfying $x_if \in I_\Delta$ for each $x_i \in U$.  If $I_\Delta = fB+A$ is a basic double G-link, the set of variables in $B$ must be an element of $\mathcal{U}_f$.  A computer computation shows that there are 60 choices of $f$ for which $\mathcal{U}_f \neq \emptyset$ and none for which any element of $\mathcal{U}_f$ has size $2$ or greater.  Thus, for any possible choice of $f$, $fB$ contains exactly one form of degree $3$.  By the arguments from previous cases, $A$ must contains all but one of the generators of $I_\Delta$, which is impossible.  

Finally, if $n\geq 4$ and $f = z_1+z_2+z_3+\cdots+z_n$, then $x_if \in I_\Delta$ implies $x_i(z_1+z_2+z_3) \in I_\Delta$, and so again $B$ has at most one monic monomial generator of degree $1$, which we have already established is impossible.  Hence, $n =1$ or $n=2$.

We next restrict the number of variables that are elements of $B$ to either one, two, or three.  For any variable $x_i$ of $B$, $x_iz_j \in C$ for $j \in [n]$.  Using the fact that $I_\Delta$ is squarefree and generated in degree $3$, no $z_j$ may be divisible by any $x_i$.  Moreover, also as above, $A$ must contain all but exactly one of the terms $x_iz_j$ for each $x_i \in B$.  Because $A \neq I_\Delta$ and $\hgt(B) = \hgt(I_\Delta) = 3$, there must be weakly between one and three variables $x_i$ in $B$. 

We will now rule out $n=1$.  Suppose $n=1$, and write $f = z_1$.  If $B$ contains exactly one variable $x_i$, then $x_iz_1$ is the only degree $3$ monomial of $fB$, and so $A$ must contain all but one of the generators of $I_\Delta$, which we have already said is impossible.  If $B$ contains exactly two variables $x_{i_1}$ and $x_{i_2}$, then all monomial generators of $I_\Delta$ other than $z_1x_{i_1}$ and $z_1x_{i_2}$ must be elements of $A$.  The condition $A:f = A$ implies that a monomial divisible by either variable dividing $z_1$ must not be a generator of $A$.  Then for all choices of $z_1$, $x_{i_1}$, and $x_{i_2}$, there is some generator of $I_\Delta$ that both must and must not be a generator of $A$, a contradiction.  The argument is the same if $B$ contains $3$ of the variables of $R$.  Hence, $n \neq 1$.

The only remaining possibility is $n=2$. Write $f =z_1+z_2$, where each $z_i$ is a product of two distinct variables.  We must consider the possibilities that $B$ contains one, two, or three variables.  These will be Subcases $2a$, $2b$, and $2c$, respectively.

\emph{Subcase} $2$\emph{a}. Assume $B$ contains exactly one variable $x_i$ of $R$.  By an argument similar to the $n=1$ case, exactly one of $x_iz_1$ and $x_iz_2$ must be a generator of $A$, and all monomial generators of $I_\Delta$ not equal to $x_iz_1$ or $x_iz_2$ must be elements of $A$. Hence, all but exactly one of the monomial generators of $I_\Delta$ must be a generator of $A$, which is impossible.

\emph{Subcase} $2$\emph{b}. Assume $B$ contains exactly two variables $x_{i_1}$ and $x_{i_2}$.  Then $A$ contains all monomial generators of $I_\Delta$ not equal to $x_{i_j}z_k$ for $j,k \in [2]$ as well as exactly one of $x_{i_1}z_1$ or $x_{i_1}z_2$ and also exactly one of $x_{i_2}z_1$ or $x_{i_2}z_2$.  There are 60 squarefree monomial ideals of height $2$ contained in $I_\Delta$ and containing all but exactly two of the generators of $I_\Delta$.  None of these ideals is Cohen--Macaulay.

\emph{Subcase} $2$\emph{c}. Assume $B$ contains exactly three variables.  Because $B$ is height $3$, $B = (x_{i_1}, x_{i_2}, x_{i_3})$.  Because each $z_jx_{i_k}$ is a degree $3$ term of $I_\Delta$, the squarefree monomial ideal, no $x_{i_k}$ can divide either $z_j$.  Because there are three distinct $x_{i_k}$ disjoint from the variables dividing the $z_j$ and only six variables in $R$, some $x_\ell$ must divide both $z_1$ and $z_2$.  Because $A:f = A$, $x_\ell$ cannot divide any generator of $A$.  However, as in previous cases, exactly one of the terms of $(z_1+z_2)x_{i_k}$ must be a term of $A$, and so we have a contradiction in this case as well.

Having exhausted all of the possibilities for the number of summands of $f$ and the number of monomial generators of $B$ of minimal degree and found a contradiction in each case, we conclude that no basic double G-link of the desired form exists.
\end{proof}

We use heavily in the argument above the fact that $A$ is squarefree monomial.  Without that assumption, there is no finite set of possible generators of $A$ nor even a degree bound on the possible homogeneous generators of $A$, and so the problem does not reduce to a finite check in any obvious way.

\section{Elementary G-biliaison and monomial ideals}\label{sect:elementaryG}

In their study of glicci simplicial complexes, Nagel and R\"omer \cite[Definition 2.2]{NagelUwe2008Gsc} introduce \textbf{squarefree glicci} simplicial complexes, which are the simplicial complexes of glicci squarefree monomial ideals that can be Gorenstein linked to monomial complete intersections via a series of links in which (at least) every second ideal is squarefree monomial.  Nagel and R\"omer give an example of a glicci monomial ideal that cannot be G-linked to a complete intersection via only other monomial ideals \cite[Example 2.1]{NagelUwe2008Gsc}.  

In the preceding sections, we studied when we might hope for (or rule out) basic double G-links between monomial ideals that do not come from weak vertex decompositions. In the cases of some key examples from the literature of simplicial complexes that are Cohen--Macaulay but not weakly vertex decomposable, we concluded there can also be no basic double G-link.  Motivated by these limitations, we use this section to consider elementary G-biliaison (see Definition \ref{def:biliais}). With an eye towards the possible future use of polarizations to move beyond the squarefree case and because \cite[Example 2.1]{NagelUwe2008Gsc} did not involve squarefree monomial ideals, we drop the squarefree requirement throughout this section.

\begin{definition}\label{def:biliais}
Let $I$, $J$, and $N$ be homogeneous, saturated, unmixed ideals of the polynomial ring $S$.  Suppose that $\hgt(I) = \hgt(J) = \hgt(N)+1$.  If $N \subseteq I \cap J$, if the localization $(S/N)_P$ is Gorenstein for every minimal prime $P$ of $S/N$, and if there exists an isomorphism $J/N \rightarrow [I/N](-\ell)$ of graded $S/N$-modules, then we say that $J$ is obtained from $I$ via an \emph{elementary G-biliaison of height $\ell$ on $N$}. 
\end{definition}

It is known (see, for example, \cite[Remark 1.13(3)]{GMN13}) that if $J$ is obtained from $I$ via an elementary G-biliaison on N, then there exists an ideal $L$ so that $L$ is a basic double G-link of $I$ on $N$ and also a basic double G-link of $J$ on $N$.  In particular, whenever an ideal $J$ is squarefree glicci via a sequence elementary G-biliaisons, there is a sequence of basic double G-links connecting $J$ to a complete intersection in which every fourth ideal is squarefree monomial.  Proposition \ref{prop:sqfreeBDLs} gives conditions that allow us to choose the ideal $L$ to be monomial.  

\begin{proposition}\label{prop:sqfreeBDLs}
Suppose that the monomial ideal $J \subset S$ is obtained from the monomial ideal $I \subset S$ via an elementary G-biliaison on $N$ with isomorphism $\varphi:J/N \rightarrow [I/N](-\ell)$.  

\begin{enumerate}
\item If $N$ is a monomial ideal and there is some $a \in J$ not a zerodivisor on $S/N$ so that either $a$ or some lift to $S$ of $\varphi(a+N)$ is a monomial, then there exists a monomial ideal $L$ that is a basic double G-link of both $J$ and $I$ on $N$.
\item If there exists some $a \in J$ not a zerodivisor on $S/N$ such that $\varphi(a+N)$ is the image in $I/N$ of $ra$ for some $r \in S$, then $I$ is a basic double G-link of $J$ on $N$.
\end{enumerate}
\end{proposition}

\begin{proof}
We first consider the form that $\varphi$ must have before examining claims $(1)$ and $(2)$ individually.

Choose any lift $x \in I$ of $\varphi(a+N)$.  Using the fact that $\varphi$ is an $S/N$-module isomorphism, we claim that $x (J/N) = a (I/N)$.   Indeed,\begin{align*}\label{isoForm}
x (J/N) &= \{\varphi(a+N) (b+N) \mid b+N \in J/N\} \\
&= \{(a+N) \varphi(b+N) \mid b+N \in J/N\} \\
&= \{a(c+N) \mid c+N \in I/N \} \\
&= a (I/N).
\end{align*}
Thus, for every $b \in J$, there exists $c \in I$ and $n \in N$ so that $xb-n = ac$. 

We claim that $\varphi$ is the map determined by multiplication by $x/a$.  Because $a$ acts invertibly on $S/N$, we may first consider images under $\varphi$ of elements of $a(J/N)$ and then, by dividing by $a$, infer the result for all elements of $J/N$. 
Fix $ab+N \in J/N$, and choose $n$ and $c$ so that $xb-n=ac$.  We compute \[
\varphi(ab+N) = b \varphi(a+N) = b(x+N) = (bx-n)+N = ac+N.
\] Hence, $\varphi(b+N) = c+N$.  

Mechanically, we may execute the map $\varphi$ on $b+N \in J/N$ by multiplying $b$ by $x$ and then modifying $bx$ by an element of $N$ to obtain a multiply of $a$ and, finally, dividing by $a$ and projecting down to $I/N$.  That is the sense in which we mean that the map $\varphi$ is multiplication by $x/a$.

Consider claim (1).  Note that because $\varphi$ is an isomorphism, $\varphi(a+N)$ is also not a zerodivisor on $S/N$. If $a$ is a monomial, then $aI+N$ is a monomial ideal of $S$, and if some lift $\widetilde{\varphi(a+N)}$ of $\varphi(a+N)$ to $S$ is a monomial, then $\widetilde{\varphi(a+N)}J+N$ is a monomial ideal.  Fix some preimage $x \in J$ of $\widetilde{\varphi(a+N)}$.

We claim that $aI+N = xJ+N$.  It suffices to show $aI \subseteq xJ+N$ and $xJ \subseteq aI+N$.  Fix $ai \in aI$.  As discussed above, there exists an element $n \in N$ so that $ai+n = xj$ with $j \in J$.  Hence, $ai+n = xj \in xJ$, and so $ai = xj-n \in xJ+N$.  The other containment is similar. Set $L = aI+N=xJ+N$.  Then the monomial ideal $L$ is a basic double G-link of both $I$ and $J$.

We now consider claim (2). Because $\varphi(a+N) = ra+N$, we have from the argument above that $ra(J/N) = a(I/N)$, and then, because $a$ acts invertibly on all of $S/N$, $r(J/N) = I/N$.  

Hence, for every element $b \in J$, there exists $c \in I$ and $n \in N$ so that $rb-n  = c$.  Because $N \subset I$, this implies that $rJ+N \subseteq I$.  But, for every $i \in I$, there exists $j \in J$ and $n \in N$ so that $i-n = rj$.  Because $N \subseteq rJ+N$, this implies that $I \subseteq rJ+N$.

The equality $I = rJ+N$ shows that $I$ is a basic double G-link of $J$ on $N$.
\end{proof}

\begin{corollary}\label{sqfreeGlicci}
Suppose that $J = I_1, \ldots, I_k = I$ is a sequence of homogeneous, saturated, unmixed monomial ideals and that $I$ is a complete intersection.  Assume that $I_i$ is obtained from $I_{i-1}$ via elementary G-biliaison for all $1 <i \leq k$ and that, in each case, the hypotheses of at least one of claim $(1)$ or  $(2)$ of Proposition \ref{prop:sqfreeBDLs} is satisfied.  Then $J$ is glicci via a sequence of basic double G-links in which all G-linked ideals are monomial.
\end{corollary}
\begin{proof}
By Proposition \ref{prop:sqfreeBDLs}, either $I_i$ is a basic double G-link of $I_{i-1}$ or $I_i$ is G-linked to $I_{i-1}$ by a sequence of two basic double G-links where the intermediary ideal $L_i$ is monomial.  Hence, $I$ is G-linked to $J$ via a sequence of direct Gorenstein links in which every other ideal (some $I_i$ or $L_i)$ is monomial.
\end{proof}

\begin{example}
We note that it is possible to have an elementary G-biliaison $\varphi:J/N \rightarrow I/N$ that does not give rise to a sequence of basic double G-links of monomial ideals connecting $J$ to $I$.  For example, if $J = (x_1,x_2,x_3)$, $I = (x_1,x_4,x_3)$, and $N = (x_1x_2-x_1x_4, x_3x_4-x_3x_2,x_1+x_3)$ in the ring $S$, then the map $\varphi$ that is multiplication by $x_4/x_2$ gives an elementary G-biliaison, but $x_4/x_2 \notin S$ and $L = x_4J+N = (x_4x_1,x_4x_2,x_4x_3, x_1x_2,x_3x_2,x_1+x_3)$ is not monomial.  Although it is easy to see that $I$ and $J$ are both basic double G-links of $(x_1,x_2x_4,x_3)$, this fact is unrelated to the elementary G-biliaison involving $\varphi$ and $N$.
\end{example}

This example gives rise to the following question:

\begin{question}
Suppose the monomial ideal $I$ is obtained from the monomial ideal $J$ via elementary G-biliaison. Must there exist sequence of monomial ideals $I = L_1, \ldots, L_k = J$ so that, for each $1<i\leq k$, either $L_i$ is a basic double G-link of $L_{i-1}$ or $L_{i-1}$ is a basic double G-link of $L_i$?
\end{question}

\section{Acknowledgements.}
Koban was partially supported by an NSERC Undergraduate Student Research Award (USRA) held at the University of Saskatchewan. Koban and Rajchgot were partially supported by NSERC grant RGPIN-2017-05732.  We thank the referee for feedback on an earlier draft of this document.

\bibliographystyle{alpha}
\bibliography{references}

\end{document}